\documentclass{amsart}

\usepackage{amssymb}
\usepackage{amsfonts}
\usepackage{amsmath}
\usepackage{amsthm}
\usepackage{graphicx}
\usepackage{mathrsfs}
\usepackage{dsfont}
\usepackage{amscd}
\usepackage{multirow}
\usepackage{mathpazo}
\usepackage[all]{xy}
\usepackage[scaled]{helvet} 
\usepackage{courier} 
\normalfont
\usepackage[T1]{fontenc}
\usepackage{calligra}
\usepackage{verbatim}
\usepackage[usenames,dvipsnames]{color}
\usepackage[colorlinks=true,linkcolor=Blue,citecolor=Violet]{hyperref}
\usepackage{enumerate}

\newcommand{\N}{{\mathds{N}}}
\newcommand{\Z}{{\mathds{Z}}}

\newcommand{\R}{{\mathds{R}}}
\newcommand{\C}{{\mathds{C}}}
\newcommand{\T}{{\mathds{T}}}

\newcommand{\D}{{\mathfrak{D}}}
\newcommand{\A}{{\mathfrak{A}}}
\newcommand{\B}{{\mathfrak{B}}}

\newcommand{\Nbar}{\overline{\N}}
\newcommand{\Lip}{{\mathsf{L}}}

\newcommand{\propinquity}[1]{{\mathsf{\Lambda}_{#1}}}
\newcommand{\dpropinquity}[1]{{\mathsf{\Lambda}^\ast_{#1}}}

\newcommand{\Kantorovich}[1]{{\mathsf{mk}_{#1}}}

\newcommand{\Haus}[1]{{\mathsf{Haus}_{#1}}}
\newcommand{\gh}{{\mathsf{GH}}}
\newcommand{\StateSpace}{{\mathscr{S}}}

\newcommand{\MongeKant}{{Mon\-ge-Kan\-to\-ro\-vich metric}}

\newcommand{\Lqcms}{{\JLL} quantum compact metric space}
\newcommand{\Qqcms}[1]{{$#1$}--\gQqcms}
\newcommand{\gQqcms}{quasi-Leibniz quantum compact metric space}

\newcommand{\unit}{1}

\newcommand{\sa}[1]{{\mathfrak{sa}\left({#1}\right)}}

\newcommand{\JLL}{Lei\-bniz}

\newcommand{\bridge}[1]{#1} 
\newcommand{\dom}[1]{{\operatorname*{dom}\left({#1}\right)}}

\newcommand{\diam}[2]{{\mathrm{diam}\left({#1},{#2}\right)}}

\newcommand{\bridgereach}[1]{{\varrho\left(\bridge{#1}\right)}}

\newcommand{\bridgeheight}[1]{{\varsigma\left(\bridge{#1}\right)}}

\newcommand{\bridgelength}[1]{{\lambda\left(\bridge{#1}\right)}}

\newcommand{\bridgenorm}[2]{{\mathsf{bn}_{ \bridge{#1}  }\left({#2}\right)}}

\newcommand{\Jordan}[2]{{{#1}\circ{#2}}} 
\newcommand{\Lie}[2]{{\left\{{#1},{#2}\right\}}} 

\newcommand{\worknote}[1]{} 

\theoremstyle{plain}
\newtheorem{theorem}{Theorem}[section]

\newtheorem{corollary}[theorem]{Corollary}

\newtheorem{lemma}[theorem]{Lemma}

\newtheorem{theorem-definition}[theorem]{Theorem-Definition}

\theoremstyle{definition}
\newtheorem{definition}[theorem]{Definition}

\theoremstyle{remark}

\newtheorem{notation}[theorem]{Notation}

\renewcommand{\geq}{\geqslant}
\renewcommand{\leq}{\leqslant}

\numberwithin{equation}{subsection}

\allowdisplaybreaks[4]

\begin{document}

\title[Full Matrix approximations]{Some applications of conditional expectations to convergence for the quantum Gromov-Hausdorff propinquity}

\author{Konrad Aguilar}
\email{konrad.aguilar@gmail.com}
\address{School of Mathematical and Statistical Sciences \\ Arizona State University \\ Tempe AZ 85281}

\author{Fr\'{e}d\'{e}ric Latr\'{e}moli\`{e}re}
\email{frederic@math.du.edu}
\urladdr{http://www.math.du.edu/\symbol{126}frederic}
\address{Department of Mathematics \\ University of Denver \\ Denver CO 80208}

\date{\today}
\subjclass[2000]{Primary:  46L89, 46L30, 58B34.}
\keywords{Noncommutative metric geometry, Gromov-Hausdorff convergence, Monge-Kantorovich distance, Quantum Metric Spaces, Lip-norms, D-norms, Hilbert modules, noncommutative connections, noncommutative Riemannian geometry, unstable $K$-theory.}

\thanks{This work is part of the project supported by the grants H2020-MSCA-RISE-2015-691246-QUANTUM DYNAMICS and the Polish Government grant 3542/H2020/2016/2}

\begin{abstract}
We prove that all the compact metric spaces are in the closure of the class of full matrix algebras for the quantum Gromov-Hausdorff propinquity. We also show that given an action of a compact metrizable group $G$ on a {\gQqcms} $(\A,\Lip)$, the function associating any closed subgroup of $G$ group to its fixed point C*-subalgebra in $A$ is continuous from the topology of the Hausdorff distance to the topology induced by the propinquity. Our techniques are inspired from our work on AF algebras as quantum metric spaces, as they are based on the use of various types of conditional expectations.
\end{abstract}
\maketitle




\section{Introduction}

The quantum Gromov-Hausdorff propinquity \cite{Latremoliere13,Latremoliere13b,Latremoliere14} provides a natural framework to discuss finite dimensional approximations of quantum spaces in a metric sense by extending the Gromov-Hausdorff distance to noncommutative geometry. Thus, for this new metric, quantum tori are limits of fuzzy tori \cite{Latremoliere13c}, spheres are limits of full matrix algebras \cite{Rieffel01,Rieffel10c,Rieffel15}, AF algebras  are limits of any inductive sequence from which they are constructed \cite{Aguilar16a, Aguilar16b, Latremoliere15d}, any separable nuclear quasi-diagonal C*-algebra equipped with a quasi-Leibniz Lip-norm is the limit of finite dimensional C*-algebras \cite{Latremoliere15c}, noncommutative solenoids are limits of matrix algebras \cite{Latremoliere16}, among other examples of such finite dimensional approximations. 

Many of these examples involve the use of a conditional expectation as a core tool. In this note, we present two new applications of conditional  expectations in constructing quantum  metrics, or proving new convergence results. We hope the techniques suggested in this work may prove helpful for future research. 

Our first new convergence result concerns full matrix algebras approximations for classical compact metric spaces. In \cite{Latremoliere13c} and \cite{Rieffel15} in particular, certain classical metric spaces are limits of full matrix algebras, an intriguing phenomenon. This note answers the natural question of which classical compact metric spaces are limits of full matrix algebras for the quantum propinquity. We shall prove that indeed, any classical compact metric space is the limit, for the quantum propinquity, of a sequence of {\Qqcms{(2,0)}s} constructed on full matrix algebras. Our approximations are very different from the ones presented in the above references, as our focus is not to preserve any symmetry of the limit space, but rather to find a very general method to obtain such full matrix algebra approximations. In particular, it is generally difficult to compute the closure of a particular set of quantum metric spaces for the propinquity. This paper proves that all classical compact metric spaces do lie in the closure of full matrix algebras for the propinquity and give examples to further test the theory of noncommutative geometry and what properties pass, or do not pass, to the limit for convergent sequences of {\gQqcms s}.

Our second new result concerns continuity of fixed point C*-subalgebras for the propinquity under certain natural assumption. If $G$ is a compact metric group acting on a quantum compact metric space $(\A,\Lip)$, then any closed subgroup of $G$ defines a fixed C*-subalgebra of $\A$. We thus have a function from the space of closed subgroups of $G$, metrized by the Hausdorff distance, to the space of fixed point C*-subalgebras of $\A$ for the action of $G$. We metrize the codomain of this map with the quantum propinquity and show that this function is indeed continuous. As an application, we obtain new results about the metric geometry of quantum tori.

We now turn to a summary of some core ingredients of noncommutative metric geometry for our current purpose.

Quantum compact metric spaces are noncommutative generalizations of Lipschitz algebras introduced in \cite{Rieffel98a,Rieffel99} by Rieffel, and inspired by Connes \cite{Connes89}. In \cite{Latremoliere13,Latremoliere15}, additional requirements were placed on the original definition of Rieffel to accommodate the construction of the quantum propinquity. The resulting notion of a {\gQqcms} will be the starting point for our work.

\begin{notation}\label{CD-quasiLeibniz-notation}
For any unital C*-algebra $\A$, we denote the unit of $\A$ by $\unit_\A$, the norm of $\A$ by $\|\cdot\|_\A$, the Jordan-Lie algebra of the self-adjoint elements of $\A$ by $\sa{\A}$, and the state space of $\A$ by $\StateSpace(\A)$.
\end{notation}

\begin{definition}[{\cite{Rieffel98a,Rieffel99,Latremoliere13,Latremoliere15}}]\label{qms-def}
A \emph{\Qqcms{F}} $(\A,\Lip)$, for some function $F:\R^4\rightarrow[0,\infty)$ weakly increasing for the product order, consists of unital C*-algebra $\A$ with unit $\unit_\A$ and a seminorm $\Lip$ defined on a dense Jordan-Lie subalgebra $\dom{\Lip}$ of the space $\sa{\A}$ of self-adjoint elements in $\A$, such that:
\begin{enumerate}
\item $\{ a \in \dom{\Lip} : \Lip(a) = 0 \} = \R\unit_\A$,
\item the \emph{\MongeKant} $\Kantorovich{\Lip}$ defined for any two states $\varphi,\psi \in \StateSpace(\A)$ by:
\begin{equation*}
\Kantorovich{\Lip}(\varphi,\psi) = \sup\left\{ \left| \varphi(a) - \psi(a) \right| : a\in\dom{\Lip}, \Lip(a) \leq 1 \right\}
\end{equation*}
metrizes the weak* topology on $\StateSpace(\A)$,
\item $\Lip$ is lower semi-continuous for $\|\cdot\|_\A$,
\item for all $a,b \in \dom{\Lip}$, we have:
\begin{equation*}
\max\left\{ \Lip\left(\Jordan{a}{b}\right), \Lip\left(\Lie{a}{b}\right)  \right\} \leq F\left(\|a\|_\A,\|b\|_\A,\Lip(a), \Lip(b)\right) \text{,}
\end{equation*}
where $\Jordan{a}{b} = \frac{ab + ba}{2}$ and $\Lie{a}{b} = \frac{ab - ba}{2i}$.
\end{enumerate}
The seminorm $\Lip$ is called an \emph{L-seminorm}.
\end{definition}

\begin{notation}
When $C\geq 1$, $D\geq  0$, and if $F : x,y,l_x,l_y \geq 0 \mapsto C(x l_y + y l_x) + D l_x l_y$, then a {\Qqcms{F}} is called $(C,D)$-quasi-Leibniz, and it is called Leibniz when $C = 1$ and $D = 0$.
\end{notation}

Rieffel provided in \cite{Rieffel98a} the fundamental characterization of compact quantum metric spaces, which is a noncommutative form of the Arz{\'e}la-Ascoli theorem. We will use a version of this characterization found in \cite{Ozawa05} in this paper, which we now recall and adapt slightly to our setting.
\begin{theorem}[{\cite{Ozawa05}}]\label{Rieffel-thm}
Let $\A$ be a unital C*-algebra, $\Lip$ a lower semi-continuous seminorm defined on some dense Jordan-Lie subalgebra $\dom{\Lip}$ of $\sa{\A}$ such that:
\begin{equation*}
\left\{ a \in \dom{\Lip} : \Lip(a) = 0 \right\} = \R \unit_\A
\end{equation*}
and, for some $C\geq 1$, $D\geq 0$:
\begin{equation*}
\max\left\{ \Lip\left(\Jordan{a}{b}\right), \Lip\left(\Lie{a}{b}\right)  \right\} \leq C\left(\|a\|_\A \Lip(b) + \|b\|_\A \Lip(a)\right) + D \Lip(a) \Lip(b)\text{.}
\end{equation*}
The following assertions are equivalent:
\begin{enumerate}
\item $(\A,\Lip)$ is a {\Qqcms{(C,D)}},
\item there exists a state $\mu \in \StateSpace(\A)$ such that the set:
\begin{equation*}
\left\{ a \in \dom{\Lip} : \mu(a) = 0, \Lip(a) \leq 1\right\}
\end{equation*}
is compact for $\|\cdot\|_\A$,
\item for all states $\mu \in \StateSpace(\A)$, the set:
\begin{equation*}
\left\{ a \in \dom{\Lip} : \mu(a) = 0, \Lip(a) \leq 1\right\}
\end{equation*}
is compact for $\|\cdot\|_\A$.
\end{enumerate}
\end{theorem}

Quasi-Leibniz quantum compact metric spaces form a category for several natural notions of morphisms \cite{Rieffel00,Latremoliere16b}. The noncompact theory is more involved \cite{Latremoliere05b,Latremoliere12b} and will not be used in this note.

Much research has been concerned with the development of a noncommutative analogue of the Gromov-Hausdorff distance, starting with the pioneering work of Rieffel in \cite{Rieffel00} on the quantum Gromov-Hausdorff distance (for which the question raised in this note was solved by the second author in \cite{Latremoliere05}). We will work with the quantum Gromov-Hausdorff propinquity introduced by Latr{\'e}moli{\`e}re in \cite{Latremoliere13} to address two inherent difficulties with the construction of such an analogue: working within a class of quantum compact metric spaces satisfying a given form of the Leibniz inequality and having the desirable property that distance zero would imply *-isomorphism of the underlying C*-algebras.

The construction of the quantum propinquity is involved, and we refer to \cite{Latremoliere13,Latremoliere13b,Latremoliere14,Latremoliere15,Latremoliere15b,Latremoliere15c,Latremoliere15d,Latremoliere16,Latremoliere16b} for a detailed discussion of this metric, its basic properties and some important applications. For our purpose, we will focus on a core ingredient of the construction of the quantum propinquity called a bridge, which enables us to appropriately relate two {\gQqcms s} and compute a quantity on which the propinquity is based.

\begin{definition}[{\cite{Latremoliere13}}]\label{bridge-def}
A \emph{bridge} $\gamma = (\D,\pi_\A,\pi_\B,x)$ from a unital C*-algebra $\A$ to a unital C*-algebra $\B$ consists of a unital C*-algebra $\D$, two unital *-monomorphisms $\pi_\A : \A \hookrightarrow \D$ and $\pi_\B : \B \hookrightarrow \D$, and an element $x \in \D$ such that:
\begin{equation*}
\StateSpace(\D|x) = \left\{ \varphi \in \StateSpace(\D) : \forall d \in \D \quad \varphi(xd) = \varphi(dx) = \varphi(d) \right\} \not= \emptyset \text{.}
\end{equation*}
\end{definition}

We associate a quantity to any bridge which estimates, for that given bridge, how far apart the domain and co-domain of the bridge are.

\begin{notation}
The Hausdorff distance \cite{Hausdorff} on the space of all compact subspaces of a metric space $(X,d)$ is denoted by $\Haus{d}$.
\end{notation}

\begin{definition}[{\cite{Latremoliere13}}]\label{bridge-length}
The \emph{length} $\bridgelength{\gamma\middle\vert\Lip_\A,\Lip_\B}$ of a bridge $\gamma = (\D,\pi_\A,\pi_\B,x)$ from $(\A,\Lip_\A)$ to $(\B,\Lip_\B)$ is the maximum of the following two quantities:
\begin{multline*}
\bridgeheight{\gamma\middle\vert\Lip_\A,\Lip_\B} = \max\left\{ \Haus{\Kantorovich{\Lip_\A}}\left(\StateSpace(\A), \left\{ \varphi\circ\pi_\A : \varphi \in \StateSpace(\D|x) \right\} \right), \right. \\ \left. \Haus{\Kantorovich{\Lip_\B}}\left(\StateSpace(\B), \left\{ \varphi\circ\pi_\B : \varphi \in \StateSpace(\D|x) \right\} \right)  \right\}
\end{multline*}
and
\begin{equation*}
\bridgereach{\gamma\middle\vert \Lip_\A,\Lip_\B} = \max\left\{ 
\begin{array}{l}
\sup_{\substack{a\in\sa{\A} \\ \Lip_\A(a) \leq 1 }} \inf_{\substack{b\in\sa{\B} \\ \Lip_\B(b) \leq 1 }} \bridgenorm{\gamma}{a,b}\\
\sup_{\substack{b\in\sa{\B} \\ \Lip_\B(b) \leq 1 }} \inf_{\substack{a\in\sa{\A} \\ \Lip_\A(a) \leq 1 }} \bridgenorm{\gamma}{a,b}
\end{array} \right\} \text{,}
\end{equation*}
where $\bridgenorm{\gamma}{a,b} = \left\|\pi_\A(a)  x - x \pi_\B(b) \right\|_\D$ for all $a\in\A$ and $b\in\B$.
\end{definition}

We note that in the present paper, all our bridges will have the unit for pivot and thus will have height zero; however the more descriptive Definition (\ref{bridge-length}) is useful to state the following characterization of the quantum propinquity which we will use as our definition for this work.

\begin{theorem-definition}[{\cite{Latremoliere13}}]\label{prop-def}
Let $F : [0,\infty)^4\rightarrow[0,\infty)$ be an increasing function for  the product order, and let $\mathfrak{QM}_{F}$ be the class of all {\Qqcms{F}s}. There exists a class function $\propinquity{F}$ on $\mathfrak{QM}_{F}\times \mathfrak{QM}_{F}$, called the \emph{quantum $F$-propinquity}, such that:
\begin{enumerate}
\item for all $(\A,\Lip_\A)$, $(\B,\Lip_\B)$ in $\mathfrak{QM}_{F}$:
\begin{multline*}
0\leq \propinquity{F}((\A,\Lip_\A),(\B,\Lip_\B)) = \propinquity{F}((\B,\Lip_\B),(\A,\Lip_\A)) \\ \leq \max\left\{ \diam{\A}{\Lip_\A},\diam{\B}{\Lip_\B}  \right\} \text{.}
\end{multline*}
\item for all $(\A,\Lip_\A)$, $(\B,\Lip_\B)$ and $(\D,\Lip_\D)$ in $\mathfrak{QM}_{F}$:
\begin{equation*}
\propinquity{F}((\A,\Lip_\A),(\D,\Lip_\B)) \leq \propinquity{F}((\A,\Lip_\A),(\B,\Lip_\B)) + \propinquity{F}((\B,\Lip_\B),(\D,\Lip_\D)) \text{,}
\end{equation*}
\item for all $(\A,\Lip_\A)$ and $(\B,\Lip_\B)$ in $\mathfrak{QM}_{F}$ and for any bridge $\gamma$ from $\A$ to $\B$, we have:
\begin{equation*}
\propinquity{F}((\A,\Lip_\A),(\B,\Lip_\B)) \leq \bridgelength{\gamma\middle\vert\Lip_\A,\Lip_\B} \text{,}
\end{equation*}
\item $\propinquity{C,D}((\A,\Lip_\A),(\B,\Lip_\B)) = 0$ if and only if there exists a *-isomorphism $\theta : \A \rightarrow \B$ such that $\Lip_\B \circ \theta = \Lip_\A$.
\end{enumerate}
Moreover, the quantum propinquity is the largest class function satisfying Assertions (1),(2), (3) and (4).
\end{theorem-definition}

\begin{notation}
When $F$ is given by Notation (\ref{CD-quasiLeibniz-notation}) for some $C\geq 1$, $D\geq 0$, then $\propinquity{F}$ is simply denoted by $\propinquity{C,D}$, and if $C=1$, $D=0$, then we may as well just write $\propinquity{}$ for $\propinquity{1,0}$.
\end{notation}

The quantum propinquity can be applied to compact metric spaces, using the following encoding of such spaces in our C*-algebraic framework --- this construction is in fact the original model for quantum compact metric spaces. We will employ the following notation all throughout this paper.
\begin{notation}
The \emph{Lipschitz seminorm} $\mathsf{Lip}_d$ for a compact metric space $(X,d)$ is defined for all functions $f \in C(X)$ by:
\begin{equation*}
\mathsf{Lip}_d (f) = \sup\left\{ \frac{|f(x) - f(y)|}{d(x,y)} : x,y \in X, x\not= y \right\} \text{,}
\end{equation*}
allowing for the value $\infty$.
\end{notation}

\begin{theorem}[{\cite{Latremoliere13}}]
If $(X,d)$ be a compact metric space, then $(C(X),\mathsf{Lip}_d)$ is a {\Lqcms}. Moreover, for all compact metric spaces $(X,d_X)$ and $(Y,d_Y)$, we have:
\begin{equation*}
\propinquity{}((C(X),\mathsf{Lip}_{d_{X}}),(C(Y),\mathsf{Lip}_{d_Y})) \leq \gh((X,d_X),(Y,d_Y)) \text{,}
\end{equation*}
where $\gh$ is the Gromov-Hausdorff distance \cite{Edwards75,Gromov81} and furthermore, the topology induced by $\propinquity{}$ on the class of classical compact quantum metric space is the same as the topology induced by $\gh$.
\end{theorem}

We now answer the question: when is a classical compact metric space the limit, not only of finite dimensional C*-algebras, but actually full matrix algebras, for the quantum propinquity?

\section{Full Matrix Approximations}

The first result of this note provides a way to construct full matrix approximations of finite metric spaces in a rather general context. 

\begin{lemma}
If $\B$ is a finite dimensional C*-subalgebra of a unital C*-algebra $\A$ and $\unit_\A \in \B$ and if $\A$ has a faithful tracial state $\mu \in \StateSpace(\A)$ then there exists a unique $\mu$-preserving conditional expectation $\mathds{E} : \A\twoheadrightarrow\B$.
\end{lemma}

\begin{proof}
See \cite[Step 1 of Theorem (3.5)]{Latremoliere15c}.
\end{proof}

\begin{theorem}\label{main-thm}
Let $(X,d)$ be a finite metric space and let:
\begin{equation*}
\delta = \min \left\{ d(x,y) : x,y \in X, x\not= y \right\} > 0\text{.}
\end{equation*}
If $\A$ is a finite dimensional C*-algebra, if $\tau$ is some faithful tracial state on $\A$, and if $\B$ is a C*-subalgebra of $\A$ such that:
\begin{enumerate}
\item $\unit_\A \in \B$,
\item there exists a unital *-isomorphism $\rho : C(X) \rightarrow \B$,
\end{enumerate}
then, for any $\beta > 0$, and setting for all $a\in\A$:
\begin{equation*}
\Lip(a) = \max\left\{ \frac{\left\| a - \mathds{E}(a)  \right\|_\A}{\beta}, \mathsf{Lip}_d\circ\rho^{-1}\left(\mathds{E}(a)\right),  \right\}
\end{equation*}
where $\mathds{E} : \A \rightarrow \B$ is the conditional expectation such that $\tau\circ\mathds{E} = \tau$, we conclude that the space $(\A,\Lip)$ is a $\left(D, 0  \right)$-quasi-Leibniz compact quantum metric space, where:
\begin{equation*}
D = \max\left\{ 2, 1 + \frac{ \beta}{\delta} \right\}
\end{equation*}
such that:
\begin{equation*}
\propinquity{}\left( (\A,\Lip), (C(X),\mathsf{Lip}_d) \right) \leq \beta \text{.}
\end{equation*}
\end{theorem}

\begin{proof}
If $a\in\A$ with $\Lip(a) = 0$ then $a = \mathds{E}(a)$, and $\mathsf{Lip}_d(\rho^{-1}(\mathds{E}(a))) = 0$, so $\mathds{E}(a) = \lambda \unit_\A$ for some $\lambda\in\R$. Thus $a \in \R \unit_\A$, as desired. We also note that $\Lip(\unit_\A) = 0$ by assumption.

We also note that since $X$ is finite, $\dom{\mathsf{Lip}_d} = C(X)$ so $\dom{\Lip} = \A$.

Since $\Lip$ is the maximum of two (lower semi-)continuous functions over $\A$, we also have $\Lip$ is (lower semi-)continuous on $\A$.

The map $\tau_X = \tau\circ\rho$ is a state of $C(X)$, and thus $\{ f \in C(X) : \tau_X(f) = 0, \mathsf{Lip}_d(f) \leq 1\}$ is compact --- since $X$ is finite, this set is actually closed and bounded in the finite dimensional space $C(X)$. Let $B > 0$ so that if $\mathsf{Lip}_d(f) \leq 1$ and $\tau_X(f) = 0$ then $\|f\|_{C(X)} \leq B$.

Now if $a\in\sa{\A}$ with $\Lip(a) \leq 1$ and $\tau(a) = 0$ then $\mathsf{Lip}_d\circ\rho^{-1}(\mathds{E}(a)) \leq 1$ and $\tau_X(\rho^{-1}(\mathds{E}(a))) = \tau\circ\mathds{E}(a) = \tau(a) = 0$. Thus $\|\mathds{E}(a)\|_\A \leq B$. Now, $\|a\|_\A \leq \|a-\mathds{E}(a)\|_\A + \|\mathds{E}(a)\|_\A \leq \beta + B$. So:
\begin{equation*}
\left\{ a \in \sa{\A} : \Lip(a) \leq 1, \tau(a) = 0 \right\} \subseteq \left\{a\in\sa{\A} : \|a\|_\A \leq \beta + B \right\} \text{,}
\end{equation*}
and the right-hand side is compact since $\A$ is finite dimensional, so $(\A,\Lip)$ is a compact quantum metric space by Theorem (\ref{Rieffel-thm}).

Last, we check the quasi-Leibniz property of $\Lip$. Let $a,b \in \dom{\Lip}$ and $x,y \in X$.  Since $\rho$ is a *-isomorphism, we now compute:
\begin{align*}
& \left| \rho^{-1}\left(\mathds{E}(ab)\right) (x) - \rho^{-1}\left(\mathds{E}(ab)\right)(y)\right|\\
&\leq
\left|\rho^{-1}\left(\mathds{E}(ab)\right) (x) - \rho^{-1}\left(\mathds{E}(a\mathds{E}(b))\right)(x)\right| \\
& \quad + \left|\rho^{-1}\left(\mathds{E}(a\mathds{E}(b)\right)) (x) - \rho^{-1}\left(\mathds{E}(\mathds{E}(a)b)\right)(y)\right| \\
& \quad + \left|\rho^{-1}\left(\mathds{E}(\mathds{E}(a)b)\right)(y)- \rho^{-1}\left(\mathds{E}(ab)\right) (y)\right|\\
&\leq\left\|\mathds{E}(a(b - \mathds{E}(b)))\right\|_\A \\
& \quad+ \left|\rho^{-1}\left(\mathds{E}(a)\right)(x) \rho^{-1}\left(\mathds{E}(b)\right) (x) - \rho^{-1}\left(\mathds{E}(a)\right)(y) \rho^{-1}\left(\mathds{E}(b)(y)\right)\right|\\
&\quad + \left\|\mathds{E}((a-\mathds{E}(a))b)\right\|_\A\\
&\leq \|a\|_\A \beta \Lip(b) + \|b\|_\A \beta\Lip(b) \\
& \quad + \left|\rho^{-1}\left(\mathds{E}(a)\right)(x) \rho^{-1}\left(\mathds{E}(b)\right) (x) - \rho^{-1}\left(\mathds{E}(a)\right)(y) \rho^{-1}\left(\mathds{E}(b)\right)(y)\right| \text{.}
\end{align*}
Hence:
\begin{equation}\label{main-L-eq}
\begin{split}
 &\mathsf{Lip}_d\circ \rho^{-1}\left(\mathds{E}(ab)\right)\\
 &= \sup\left\{ \frac{\left| \rho^{-1}\left(\mathds{E}(ab)\right) (x) - \rho^{-1}\left(\mathds{E}(ab)\right)(y)\right|}{d(x,y)} : x,y \in X , x\not= y \right\} \\
&\leq \|a\|_\A \frac{\beta}{\delta} \Lip(b) + \|b\|_\A \frac{\beta}{\delta}\Lip(b) \\
&\quad + \sup \Bigg\{ \frac{ \left|\rho^{-1}\left(\mathds{E}(a)\right)(x) \rho^{-1}\left(\mathds{E}(b)\right) (x) - \rho^{-1}\left(\mathds{E}(a)\right)(y) \rho^{-1}\left(\mathds{E}(b)\right)(y)\right|}{d(x,y)}\\
& \quad \quad \quad \quad \quad \quad \quad \quad \quad : x,y\in X, x\not=y \Bigg\}\\
&\leq \frac{\beta}{\delta} \left(\|a\|_\A\Lip(b) + \Lip(a)\|b\|_\A\right) + \mathsf{Lip}_d(\mathds{E}(a) \mathds{E}(b)) \\
&\leq \frac{\beta}{\delta} \left(\|a\|_\A\Lip(b) + \Lip(a)\|b\|_\A\right) + \mathsf{Lip}_d\circ\mathds{E}(a) \|b\|_\A + \|a\|_\A \mathsf{Lip}_d\circ\mathds{E}(b) \\
&\leq \left(1 + \frac{\beta}{\delta}\right)\left(\|a\|_\A \Lip(b) + \Lip(a) \|b\|_\A\right) \text{.}
\end{split}
\end{equation}
From this and from \cite[Lemma 3.2]{Latremoliere15d}, it follows easily that $(\A,\Lip)$ is indeed a {\Qqcms{(D, 0)}} with $D = \max\left\{2,\left(1+\frac{\beta}{\delta}\right)\right\}$.

We now compute an upper bound for $\propinquity{}((\A,\Lip),(C(X),\mathsf{Lip}_d))$ by exhibiting a particular bridge from $\A$ to $C(X)$.

Let $\gamma = (\A, \mathrm{id}, \rho, \unit_\A)$ where $\mathrm{id}$ is the identity *-morphism of $\A$. By Definition (\ref{bridge-def}), the quadruple $\gamma$ is a bridge of height $0$, so its length equals to its reach.

If $f \in C(X)$ and $\mathsf{Lip}_d(f) \leq 1$, then:
\begin{equation*}
\frac{\left\| \rho(f) - \mathds{E}(\rho(f)) \right\|_\A}{\beta} = 0
\end{equation*}
and $\mathsf{Lip}_d(\rho^{-1}(\mathds{E}(\rho(f)))) = \mathsf{Lip}_d(f) \leq 1$. So $\Lip(\rho(f)) \leq 1$.

Now, it is immediate that $\bridgenorm{\gamma}{\rho(f),f} = \|\rho(f) - \rho(f)\|_\A = 0$. So:
\begin{equation*}
\sup_{\substack{f\in C(X)\\ \mathsf{Lip}_d(f) \leq 1}} \inf_{\substack{a\in\sa{\A}\\\Lip(a) \leq 1}} \bridgenorm{\gamma}{a,b} = 0 \text{.}
\end{equation*}

If $a \in \A$ with $\Lip(a) \leq 1$, then set $f = \rho^{-1}\left(\mathds{E}(a)\right))$. First, by definition of $\Lip$, we have $\mathsf{Lip}_d(f) = \mathsf{Lip}_d(\rho^{-1}(\mathds{E}(a))) \leq \Lip(a) \leq 1$. Second:
\begin{equation*}
\left\| a - \rho(f) \right\|_\A = \left\| a - \mathds{E}(a) \right\|_\A \leq \beta\text{.}
\end{equation*}
Thus
\begin{equation*}
\sup_{\substack{a\in\sa{\A}\\ \Lip(a) \leq 1}} \inf_{\substack{f\in C(X)\\ \mathsf{Lip}_d(f) \leq 1}} \bridgenorm{\gamma}{a,b} \leq \beta \text{.}
\end{equation*}

Therefore, the reach, and thus the length, of $\gamma$ is no more than $\beta$. Hence by Theorem-Definition (\ref{prop-def}), we conclude $\propinquity{}((\A,\Lip),(C(X),\mathsf{Lip}_d)) \leq \beta$ as desired.
\end{proof}

We now deduce from Theorem (\ref{main-thm}) that compact metric spaces are always limits of full matrix algebras for the quantum propinquity. A notable component of the following result is how the constant $\beta$ of Theorem (\ref{main-thm}) are related to the actual geometry of the limit classical space.

\begin{corollary}\label{main-cor}
If $(X,d)$ is a compact metric space, if $Y \subseteq X$ is a finite subset of $X$, and if $\beta_Y \in (0,\infty)$ such that:
\begin{equation*}
\frac{\beta_Y}{\min \{ d(x,y) : x,y \in Y, x\not= y \} } \leq 1
\end{equation*}
then there exists a {\Qqcms{(2,0)}}  $(\A,\Lip)$   where:
\begin{enumerate}
\item $\A$ is the C*-algebra of $\# Y \times \# Y$-matrices over $\C$ and $\tau$ is the unique tracial state on $\A$,
\item with $C(Y)$ identified with the diagonal C*-subalgebra of $\A$ given by a unital *-isomorphism $\rho$ with domain $C(Y)$ and $\mathds{E}_Y$, the unique $\tau$-preserving conditional expectation of $\A$ onto $\rho(C(Y))$, the L-seminorm $\Lip$ is given for all $a\in\A$ by:
\begin{equation}\label{main-cor-eq1}
\Lip(a) = \max\left\{ \frac{\left\|a -\mathds{E}_Y(a)\right\|_\A}{\beta_Y}, \mathsf{Lip}_d \circ \rho^{-1} \left(\mathds{E}_Y(a)\right) \right\}\text{,}
\end{equation} and 
\item $\propinquity{}((\A,\Lip), (C(X),\mathsf{Lip}_d)) \leq \Haus{d}(X,Y) + \beta_Y$. 
\end{enumerate}
\end{corollary}

\begin{proof}
Set $\delta = \min \{ d(x,y) : x,y \in Y, x\not= y \} $. 
By Theorem (\ref{main-thm}), the compact quantum metric space $(\A,\Lip)$ is$(2,0)$-quasi-Leibniz since $1 + \frac{\beta_Y}{\delta} \leq 2$ and:
\begin{equation*}
\propinquity{}((\A,\Lip),(C(Y),\mathsf{Lip}_d)) \leq \beta_Y\text{.}
\end{equation*}
Thus:
\begin{multline*}
\propinquity{}((\A,\Lip),(C(X),\mathsf{Lip}_d)) \leq \\ \propinquity{}((\A,\Lip),(C(Y),\mathsf{Lip}_d)) + \propinquity{}((C(Y),\mathsf{Lip}_d),((C(X),\mathsf{Lip}_d))) \\ \leq  \beta_Y + \Haus{d}(X,Y) \text{.}
\end{multline*}
This concludes our proof.
\end{proof}

\begin{corollary}
Any compact metric space $(X,d)$ is the limit for the quantum propinquity of sequences of {\Qqcms{(2,0)}s} consisting of full matrix algebras.
\end{corollary}

\begin{proof}
We simply apply Corollary (\ref{main-cor}) to any sequence $(X_n)_{n\in\N}$ of finite subsets of $X$ with $\lim_{n\rightarrow\infty}\Haus{d}(X,X_n) = 0$, which always exists since $(X,d)$ is compact, and to $\left(\beta_{X_n}\right)_{n\in\N} = \left(\frac{\min\{d(x,y) : x,y \in X_n, x\not=y\}}{n}\right)_{n\in\N}$.
\end{proof}

\section{Fixed Point C*-subalgebras}

We now turn to the second result of this note. We employ conditional expectations again  as a key tool, though this time, our conditional expectations are constructed via group actions and are not used in the definition of the quantum metrics, unlike the previous section. In this section, we prove a continuity result for quantum metric spaces constructed as fixed point C*-subalgebras of some given {\gQqcms}, for some fixed compact group action. We refer to \cite{Latremoliere17b} for more results regarding group actions and {\gQqcms s}.

\begin{theorem}\label{fixed-point-propinquity-thm}
Let $(\A,\Lip)$ be a {\Qqcms{F}} for some admissible function $F$. Let $G$ be a compact metrizable group endowed with a continuous length function $\ell$. Let $\alpha$ be a strongly continuous action of $G$ by *-automorphisms on $\A$ such that $\Lip\circ\alpha^g\leq\Lip$ for all $g\in G$.

If $H$ is any closed subgroup of $G$, let $\A_H = \{ a \in \A : \forall g \in H \quad \alpha^h(a) = a \}$ be the fixed C*-subalgebra of $\A$ for the restriction of the action $\alpha$ to $H$.

If $(G_n)_{n\in\N}$ is a sequence of closed subgroups of $G$ converging to $G_\infty$ for the Hausdorff distance $\Haus{\ell}$, then:
\begin{equation*}
\lim_{n\rightarrow\infty} \dpropinquity{F}\left((\A_{G_n},\Lip),(\A_{G_\infty})\right) = 0 \text{.}
\end{equation*}
\end{theorem}

\begin{proof}
Let $\Nbar = \N\cup\{\infty\}$. For each $n\in\Nbar$, let $\lambda_n$ be the left Haar probability measure on $G_n$. By \cite[Lemma 3.6]{Latremoliere05}, the sequence $(\lambda_n)_{n\in\N}$ weak* converges to $\lambda_\infty$ as \emph{measures over $G$} (where $\lambda_n$ is identified with $\lambda_n(\cdot\cap G_n)$ for all $n\in\N$), i.e. if $f : G \rightarrow \R$ is a continuous function, then:
\begin{equation*}
\lim_{n\rightarrow\infty} \int_G f(g) \,d\lambda_n(g) = \int_{G} f(g) \,d\lambda_\infty(g) \text{.}
\end{equation*}

We define, for all $n\in\Nbar$ and $a\in\A$:
\begin{align*}
\mathds{E}_n(a)  &= \int_{G} \alpha^g(a) \, d\lambda_n(g)
\end{align*}
and, as is well-known and easily checked, $\mathds{E}_n$ is a conditional expectation of $\A$ onto $\A_{G_n}$.

We note that for all $n\in\Nbar$ and $a\in\sa{\A}$:
\begin{equation*}
\Lip(\mathds{E}_n(a)) \leq \int_{G} \Lip(\alpha^g(a)) \, d\lambda_n(g) \leq \int_G \Lip(a) \,d\lambda_n(g) = \Lip(a) \text{,}
\end{equation*}
as $\Lip$ is lower semi-continuous.

In particular, let $a\in\sa{\A_{G_n}}$ and $\varepsilon > 0$. Note that $a = \mathds{E}_n(a)$. On the other hand, by definition, there exists $b\in\dom{\Lip}$ with $\|a-b\|_\A < \varepsilon$. Therefore:
\begin{equation*}
\|a-\mathds{E}_n(b)\|_\A = \|\mathds{E}_n(a-b)\|_\A \leq \|a-b\|_\A < \varepsilon
\end{equation*}
and we note that $\Lip(\mathds{E}_n(b)) \leq \Lip(b) < \infty$, so $\mathds{E}_n(b) \in \dom{\Lip}\cap\A_{G_n}$. Hence $\dom{\Lip}\cap\A_{G_n}$ is dense in $\sa{\A_{G_n}}$ since $\varepsilon > 0$ was arbitrary. It then easily follows that $(\A_{G_n},\Lip)$ is a {\Qqcms{F}} (where we keep the notation $\Lip$ for the restriction of $\Lip$ to $\A_{G_n}$).

We now establish the convergence of the fixed point C*-algebras. Fix any $\mu\in\StateSpace(\A)$. Let $\varepsilon > 0$. Let $\mathcal{F}$ be a $\frac{\varepsilon}{5}$-dense finite subset of $\StateSpace(\A)$ for $\Kantorovich{\Lip}$ (note: $\mu$ need not be in $\mathcal{F}$). Let $\mathcal{A}$ be a finite $\frac{\varepsilon}{5}$-dense subset of $\{a\in\dom{\Lip}:\Lip(a)\leq  1, \mu(a) = 0\}$ for $\|\cdot\|_\A$.

For each $\varphi\in\mathcal{F}$, and $a\in\mathcal{A}$, let $N_{a,\varphi} \in \N$ such that for all $n\geq N_{a,\varphi}$, we have:
\begin{equation*}
\left| \int_G \varphi(\alpha^g(a)) \, d\lambda_n(g) - \int_G \varphi(\alpha^g(a)) \, d\lambda_\infty(g)\right| < \frac{\varepsilon}{5} \text{.}
\end{equation*}

Let $N = \max\{N_{a,\varphi} : a\in\mathcal{A}, \varphi \in \mathcal{F} \}$ and $n\geq N$. Let $\varphi \in \mathcal{F}$. Let $a\in\mathcal{A}$. We then compute:

\begin{align*}
\left|\varphi\left( \mathds{E}_n(a) - \mathds{E}_\infty(a) \right)\right| &= \left|\varphi\left( \int_{G} \alpha^g(a) \,d\lambda_n(g) - \int_{G} \alpha^g(a) \,d\lambda_\infty(g) \right)\right| \\
&= \left|\int_{G} \varphi\left(\alpha^g(a)\right) \,d\lambda_n(g) - \int_{G} \varphi\left(\alpha^g(a)\right) \,d\lambda_\infty(g)\right|  \\
&\leq \frac{\varepsilon}{5} \text{.}
\end{align*}

Let $\psi\in\StateSpace(\A)$ and $\varphi\in\mathcal{F}$ such that $\Kantorovich{\Lip}(\varphi,\psi) < \varepsilon$. Let $a\in\sa{\A}$ with $\Lip(a)\leq 1$. Since $\mu(a-\mu(a)\unit_\A) =0$ and $\Lip(a-\mu(a)\unit_\A) = \Lip(a) \leq 1$, there exists $b\in\mathcal{A}$ such that $\|(a-\mu(a)\unit_\A)-b\|_\A <\frac{\varepsilon}{5}$. Now, since $\Lip(\mathds{E}_n(a)) \leq \Lip(a) \leq 1$, we have:
\begin{align*}
\left|\psi\left( \mathds{E}_n(a) - \mathds{E}_\infty(a)\right)\right| &= \left|\psi\left( \mathds{E}_n(a)\right) - \psi\left(\mathds{E}_\infty(a)\right) \right|\\
&\leq\left|\psi\left( \mathds{E}_n(a)\right) - \varphi\left(\mathds{E}_n(a)\right)\right| \\
&\quad + \left| \varphi\left(\mathds{E}_n(a)\right) - \varphi\left(\mathds{E}_\infty(a)\right) \right| \\
&\quad + \left| \varphi\left( \mathds{E}_\infty(a) \right) - \psi\left(\mathds{E}_\infty(a)\right) \right|\\
&\leq 2\frac{\varepsilon}{5} + \left| \varphi\left(\mathds{E}_n(a)\right) - \varphi\left(\mathds{E}_\infty(a)\right) \right|\\
&= 2\frac{\varepsilon}{5}  + \left| \varphi\left( \mathds{E}_n(a-\mu(a)\unit_\A) - \mathds{E}_\infty(a-\mu(a)\unit_\A) \right) \right| \\
&\leq 2\frac{\varepsilon}{5} + \left|\varphi\left( \mathds{E}_n(a-\mu(a)\unit_\A) - \mathds{E}_n(b)\right)\right| \\
&\quad + \left| \varphi\left( \mathds{E}_n(b) - \mathds{E}_\infty(b)\right)\right| \\
&\quad + \left| \varphi\left( \mathds{E}_\infty(b) - \mathds{E}_\infty(a-\mu(a)\unit_\A)\right) \right| \\
&\leq 2\frac{\varepsilon}{5} + 2\|(a-\mu(a)\unit_\A)-b\|_\A + \left| \varphi\left( \mathds{E}_n(b) - \mathds{E}_\infty(b)\right) \right| \\
&\leq 4\frac{\varepsilon}{5} + \frac{\varepsilon}{5} = \varepsilon \text{.}
\end{align*}

Thus, for all $a\in\sa{\A}$ with $\Lip(a) \leq 1$ and for all $n\geq N$, since $\mathds{E}_n(a)-\mathds{E}_\infty(a)$ is self-adjoint, we have:
\begin{equation*}
\left\|\mathds{E}_n(a) - \mathds{E}_\infty(a)\right\|_\A \leq \varepsilon \text{.}
\end{equation*}

We now work with the bridge $\gamma = (\A,\iota_n,\iota_\infty,\unit_\A)$ where $\iota_n : \A_{G_n}\hookrightarrow \A$ is the canonical injection for all $n\in\Nbar$. As the pivot of this bridge is the unit, this bridge has height $0$.

Now, let $a\in\sa{\A_{G_n}}$ (so $a=\mathds{E}_n(a)$) with $\Lip(a) \leq 1$ for $n\geq N$. We compute:
\begin{equation*}
\left\| a - \mathds{E}_\infty(a) \right\|_\A = \left\| \mathds{E}_n(a) - \mathds{E}_\infty(a) \right\|_\A \leq \varepsilon \text{.}
\end{equation*}
If $a\in\sa{\A_{G_\infty}}$ with $\Lip(a) \leq 1$ and $n\geq N$ then:
\begin{equation*}
\left\| a - \mathds{E}_n(a) \right\|_\A = \left\| \mathds{E}_\infty(a) - \mathds{E}_n(a) \right\|_\A \leq \varepsilon \text{.}
\end{equation*}

Hence, the reach of the bridge $\gamma$ is no more than $\varepsilon$.
\end{proof}

We can apply Theorem (\ref{fixed-point-propinquity-thm}) for various new convergence results. 

\begin{corollary}\label{qt-cor}
Let $\sigma$ be a multiplier of $\Z^d$, with $d\in\N\setminus\{0,1\}$. Let $\ell$ be a continuous length function on $\T^d = \left\{ (z_1,\ldots,z_d) \in \C^d : \forall j \in\{1,\ldots,d\} \quad |z_j| = 1\right\}$. We denote the dual action of $\T^d$ on the quantum torus $\A_\sigma = C^\ast(\Z^d,\sigma)$ by $\alpha$. For any closed subgroup $G$ of $\T^d$, we denote the fixed point C*-subalgebra of $\A_\sigma$ for $\alpha$ restricted to $G$ as $\A_\sigma^G$.

For all $a\in\A_\sigma$, we set:
\begin{equation*}
\Lip(a) = \sup\left\{ \frac{\|a-\alpha^z(a)\|_{\A_\sigma}}{\ell(z)} : z \in \T^d\setminus\{(1,\ldots,1)\} \right\}\text{.}
\end{equation*}

If $(G_n)_{n\in\N}$ is a sequence of closed subgroups of $\T^d$ converging to some closed subgroup $G_\infty$ of $\T^d$ for the Hausdorff distance $\Haus{\ell}$ induced by the invariant metric defined by $\ell$ on $\T^d$, then:
\begin{equation*}
\lim_{n\rightarrow\infty} \dpropinquity{}\left((\A_\sigma^{G_n},\Lip), (\A_\sigma^{G_\infty},\Lip)\right) = 0 \text{.}
\end{equation*}
\end{corollary}

\begin{proof}
The seminorm $\Lip$ is a Leibniz L-seminorm, as shown in \cite{Rieffel98a}, and by construction $\Lip\circ\alpha^z = \Lip$ for all $z\in\T^d$. Thus, we are in the setting of Theorem (\ref{fixed-point-propinquity-thm}), and the conclusion follows.
\end{proof}

Corollary (\ref{qt-cor}) differs from \cite[Theorem 4.4]{Latremoliere05} and its version for the propinquity \cite{Latremoliere15b} as the continuous length function involved in our new corollary is fixed, unlike \cite[Theorem 4.4]{Latremoliere05}, and thus the convergence result is not due to changing the geometry of the torus, but rather by averaging over a convergent sequence of closed subgroups.

\providecommand{\bysame}{\leavevmode\hbox to3em{\hrulefill}\thinspace}
\providecommand{\MR}{\relax\ifhmode\unskip\space\fi MR }
\providecommand{\MRhref}[2]{%
  \href{http://www.ams.org/mathscinet-getitem?mr=#1}{#2}
}
\providecommand{\href}[2]{#2}

\vfill

\end{document}